\documentclass{amsart}

\usepackage{amsmath}
\usepackage{amssymb}
\usepackage{indentfirst}
\usepackage{galois}
\usepackage{amsthm}
\usepackage{amsfonts}

\newtheorem{theorem}{Theorem}[section]
\newtheorem{proposition}[theorem]{Proposition}
\newtheorem{lemma}[theorem]{Lemma}
\newtheorem{corollary}[theorem]{Corollary}
\newtheorem{claim}[theorem]{Claim}

\newtheorem*{question1}{Question A}
\newtheorem*{question2}{Question B}
\newtheorem*{main1}{Main Result A}
\newtheorem*{main2}{Main Result B}
\newtheorem*{open}{Open Question}

\theoremstyle{definition}
\newtheorem{definition}[theorem]{Definition}

\theoremstyle{remark}
\newtheorem{remark}[theorem]{Remark}

\numberwithin{equation}{section}

\begin{document}

\bibliographystyle{plain}
\title{Complex symmetric composition operators induced by linear fractional maps}

\author[Y.X. Gao and Z.H. Zhou] {Yong-Xin Gao and Ze-Hua Zhou$^*$}
\address{\newline  Yong-Xin Gao\newline School of Mathematics,
Tianjin University, Tianjin 300350, P.R. China.}

\email{tqlgao@163.com}

\address{\newline Ze-Hua Zhou\newline School of Mathematics, Tianjin University, Tianjin 300350,
P.R. China.}
\email{zehuazhoumath@aliyun.com; zhzhou@tju.edu.cn}


\date{}
\thanks{\noindent $^*$Corresponding author.\\
This work was supported in part by the National Natural Science
Foundation of China (Grant Nos. 11371276; 11301373; 11401426).}

\begin{abstract}
In this paper we give the answers to two open questions on complex symmetric composition operators. By doing this, we give a complete description of complex symmetric composition operators whose symbols are linear fractional.
\end{abstract}

\maketitle

\section{Introduction}

Let $T$ be a bounded operator on a complex Hilbert space $\mathcal{H}$. Then $T$ is called complex symmetric if there exists a conjugation $C$ such that $T=CT^*C$. Here a conjugation is a conjugate-linear, isometric involution on $\mathcal{H}$. Precisely speaking, a mapping $C:\mathcal{H}\to\mathcal{H}$ is called a conjugation on $\mathcal{H}$ if it satisfies the following conditions,

(i) $C(\lambda x+\mu y)=\overline{\lambda}Cx+\overline{\mu}Cy$ for all $x,y\in\mathcal{H}$ and $\lambda,\mu\in\mathbb{C}$;

(ii) $\langle Cx,Cy\rangle=\langle y,x\rangle$ for all $x,y\in\mathcal{H}$;

(iii) $C^2=I$ is the identity on $\mathcal{H}$.

The operator $T$ may also be called $C$-symmetric if $T$ is complex symmetric with respect to a specifical conjugation $C$. For more details about complex symmetric operators one may turn to \cite{CS1}, \cite{Cs2}, and \cite{Cs3}.

In this paper, we are particularly interested in the complex symmetry of composition operators induced by analytic self-maps of $D$. This subject was started by Garcia and Hammond in \cite{GH}.

Recall that for each analytic self-map $\varphi$ of the unit disk $D$, the composition operator $C_\varphi$ given by
$$C_\varphi f=f\comp\varphi$$
is always bounded on $H^2(D)$. Here, the Hardy-Hilbert space $H^2(D)$ is the set of analytic functions on $D$ such that
$$||f||^2=\sup_{0<r<1}\int_{0}^{2\pi}|f(re^{i\theta})|^2\frac{d\theta}{2\pi}<\infty.$$

Several simple examples of complex symmetric composition operators on $H^2(D)$ arise immediately. For example, every normal operator is complex symmetric ( see \cite{CS1}), so when $\varphi(z)=sz$ with $|s|\leqslant 1$, $C_\varphi$ is normal hence complex symmetric on $H^2(D)$. Also, Theorem 2 in \cite{Cs3} states that
each operator satisfying a polynomial equation of order $2$ is complex symmetric. So when $\varphi$ is an involutive automorphism, we have $C_\varphi^2=I$, thus $C_\varphi$ is complex symmetric on $H^2(D)$. In \cite{Invo} Waleed Noor found the conjugation $C$ such that $C_\varphi$ is $C$-symmetric when $\varphi$ is an involutive automorphism.

However, these examples are trivial to some extent. So people wondered if there exists any other example of complex symmetric composition operator on $H^2(D)$.
The first step is, of course, to determine the complex symmetric composition operators induced by automorphisms. Bourdon and Waleed Noor considered this problem in \cite{BN}. The next proposition is Proposition 2.1 in \cite{BN}.

\begin{proposition}\label{propa}
Let $\varphi$ be a self-map of $D$. If $C_\varphi$ is complex symmetric on $H^2(D)$, then $\varphi$ has a fixed point in $D$.

Particularly, if $\varphi$ is either a parabolic or a hyperbolic automorphism, then $C_\varphi$ cannot be complex symmetric on $H^2(D)$.
\end{proposition}

Thanks to this proposition, we only need to investigate the elliptic automorphisms of the unit disk $D$. It turns out that things depend much on the orders of the automorphisms. Most of the cases are settled by the following proposition, which is one of the main results in \cite{BN}.

\begin{proposition}\label{propb}
Let $\varphi$ be an elliptic automorphism of order $q$. If $q=2$, then $C_\varphi$ is always complex symmetric on $H^2(D)$. If $4\leqslant q\leqslant\infty$, then $C_\varphi$ is complex symmetric on $H^2(D)$ only if $\varphi$ is a rotation.
\end{proposition}

However, the order 3 elliptic case remains as an open question, which is posed by Bourdon and Waleed Noor in \cite{BN}:

\begin{question1}
Is $C_\varphi$ complex symmetric on $H^2(D)$ when $\varphi$ is an elliptic automorphism of order $3$?
\end{question1}

We will give the answer to this question in the first part of this paper. By doing this we complete the project of finding all invertible composition operators which are complex symmetric on $H^2(D)$. More precisely, we will prove that if $\varphi$ is an elliptic automorphism of order $3$ and not a rotation, then the composition operator $C_\varphi$ cannot be complex symmetric on $H^2(D)$. Thus we can get our Main Result A as follows. It will be given as Corollary \ref{last} in the third section.

\begin{main1}
Suppose $\varphi$ is an automorphism of the unit disk $D$. Then $C_\varphi$ is complex symmetric on $H^2(D)$ if and only if $\varphi$ is either a rotation or an involutive automorphism.
\end{main1}

This result shows that one can never find any nontrivial example of complex symmetric composition operator among the automorphisms of $D$. So we continue to consider the non-automorphisms.

Sungeun Jung et al \cite{JFA} posed a non-automorphic example of composition operator which was thought to be complex symmetric, but then it was disproved by Waleed Noor in \cite{Cor}. Recently, Narayan et al \cite{Arxiv} gave the first non-automorphic examples of complex symmetric composition operators. Their examples are as follows, which is Theorem 2.10 in \cite{Arxiv}.

\begin{proposition}\label{state}
Let $\varphi_1(z)=az+c$ and $\varphi_2(z)=az/(1-cz)$ be analytic self-maps of $D$ and neither of them is an automorphism of $D$. Then $C_{\varphi_1}$, respectively $C_{\varphi_2}$, is complex symmetric on $H^2(D)$ if and only if $\varphi_1$, respectively $\varphi_2$, has no fixed point on the boundary of $D$.
\end{proposition}

Yet their examples are linear fractional. Then they asked at the last part of their paper: What about the other linear fractional self-maps of $D$? Can any of them induces a complex symmetric composition operator?

\begin{question2}
Find all the linear fractional self-maps of $D$ that can induce complex symmetric composition operators on $H^2(D)$.
\end{question2}

In this paper, we also answer this question by giving a complete description of complex symmetric composition operators whose symbols are linear fractional. More precisely, we show that the examples in \cite{Arxiv} are the only non-automorphic examples that can be found among the linear fractional maps. Our second main result is as follows, which is Theorem \ref{main2} in the fourth section.

\begin{main2}
Suppose $\varphi$ is a linear fractional self-maps of $D$ and is not a constant. Then $C_\varphi$ is complex symmetric on $H^2(D)$ if and only if at least one of the following conditions holds,
\\
(i) the two fixed points of $\varphi$ on $\hat{\mathbb{C}}$ are $0$ and a point outside $\overline{D}$;
\\
(ii) the two fixed points of $\varphi$ on $\hat{\mathbb{C}}$ are $\infty$ and a point in $D$;
\\
(iii) $\varphi$ is an involutive automorphism.
\end{main2}

\section{Preliminary}

The space we considered throughout this paper is the Hardy space $H^2(D)$. It is a Hilbert space, with the inner product
$$\langle f,g\rangle=\sup_{0<r<1}\int_{0}^{2\pi}f(re^{i\theta})\overline{g(re^{i\theta})}\frac{d\theta}{2\pi}.$$
For each $w\in D$, let
$$K_w(z)=\frac{1}{1-\overline{w}z}.$$
Then $K_w\in H^2(D)$ is the evaluation functional at the point $w$, i.e., $$\langle f,K_w\rangle=f(w)$$ for all $f\in H^2(D)$. Futhermore, for each $w\in D$ and every integer $j>0$, we can find a unique function $K_w^{[j]}\in H^2(D)$ such that
$$\langle f,K_w^{[j]}\rangle=f^{(j)}(w)$$ for all $f\in H^2(D)$. The function $K_w^{[j]}$ is called evaluation of the $j$-th derivative at $w$.
\\

A linear fractional transformation is a map of the form
$$\varphi(z)=\frac{az+b}{cz+d}.$$
We always assume $ad-bc\ne 0$, so that $\varphi$ is not a constant. Each linear fractional transformation can be regarded as a biholomorphic mapping of $\hat{\mathbb{C}}$ onto itself. Here $\hat{\mathbb{C}}=\mathbb{C}\cup\{\infty\}$ is the Riemann sphere. We denote the set of all linear fractional transformations by $LFT(\hat{\mathbb{C}})$. Recall that each linear fractional transformation has at most two fixed points on $\hat{\mathbb{C}}$.

By $LFT(D)$ we denote the set of all linear fractional self-maps of the unit disk $D$. It is a subgroup of $LFT(\hat{\mathbb{C}})$. The automorphisms of $D$ are all contained in $LFT(D)$.

It is well known that the automorphisms of $D$ fall into three categories:

$\bullet$ An automorphism $\varphi$ is called elliptic if it has a unique fixed point in $D$.

$\bullet$  $\varphi$ is called hyperbolic if it has no fixed point in $D$ and two distinct fixed points on the unit circle $\partial D$.

$\bullet$ $\varphi$ is called parabolic if it has no fixed point in $D$ and one fixed point on $\partial D$.

Proposition \ref{propa} tells us that when $\varphi$ is either hyperbolic or parabolic, the composition operator $C_\varphi$ cannot be complex symmetric on $H^2(D)$. So we only need to take care of the elliptic ones.

\begin{definition}
The order of an elliptic automorphism $\varphi$ is the smallest positive integer such that $\varphi_n(z)=z$ for all $z\in D$. Here $\varphi_n=\varphi\comp\varphi\comp...\comp\varphi$ denotes the $n$-th iterate of $\varphi$. If no such positive integer exists, then $\varphi$ is said to have order $\infty$.
\end{definition}

Note that if the order of an automorphism $\varphi$ is one, then $\varphi$ is identity on $D$. If $\varphi$ has order two, then $\varphi$ is of the form
$$\varphi(z)=\varphi_a(z)=\frac{a-z}{1-\overline{a}z}$$
for some $a\in D$.
\begin{remark}
The notation $\varphi_a$ will be used throughout this paper. $\varphi_a$ is the involution automorphism exchanges $0$ and $a$.
\end{remark}

For the convenience of our discussion, we restate Proposition \ref{state}, i.e., Theorem 2.10 in \cite{Arxiv}, as follows.

\begin{proposition}\label{he}
Suppose $\varphi\in LFT(D)$ is not an automorphism and either $0$ or $\infty$ is a fixed point of $\varphi$. Then $C_\varphi$ is complex symmetric on $H^2(D)$ if and only if $\varphi$ has no fixed point on the unit circle $\partial D$.
\end{proposition}

The next lemma is a simple property of complex symmetric operators. It will be used repeatedly in this paper.

\begin{lemma}\label{zh}
Suppose $T$ is complex symmetric on $\mathcal{H}$ with respect to a conjugation $C$, then $\lambda\in\mathbb{C}$ is a eigenvalue of $T$ if and only if $\overline{\lambda}$ is a eigenvalue of $T^*$. Moreover, the conjugation $C$ maps the eigenvectors subspace $\mathrm{Ker}(T-\lambda)$ onto $\mathrm{Ker}(T^*-\overline{\lambda})$.
\end{lemma}

\begin{proof}
One only need to note that $T=CT^*C$ implies  $T-\lambda=C(T^*-\overline{\lambda})C$.
\\
\qedhere
\end{proof}

The next lemma follows directly from the proof of Lemma 2.2 in \cite{BN}.

\begin{lemma}\label{*s}
On $H^2(D)$ we have $C_{\varphi_a}^*1=K_a$ and $C_{\varphi_a}^*z^{k+1}=e_{k+1}-ae_k$. Here $e_k=K_a\varphi_a^k$ for each positive integer $k$.
\end{lemma}

As a corollary, we can get the eigenvectors of $C_\varphi$ on $H^2(D)$ when $\varphi$ is an elliptic automorphism of order $3$. %The main calculation of this lemma is done in \cite{BN}, with the help of Theorem 9.2 in Cowen and MacCluer's book \cite{CC}.

\begin{corollary}\label{le}
Suppose $\varphi$ is an elliptic automorphism of order $3$ with fixed point $a\in D$. Let $\Lambda_m=\mathrm{Ker}(C_\varphi-\varphi'(a)^m)$ and $\Lambda^*_m=\mathrm{Ker}(C_\varphi^*-\overline{\varphi'(a)}^m)$ for $m=0,1,2$, then
$$\Lambda_m=\overline{ \mathrm{span}}\{\varphi_a^{3j+m}; j=0,1,2,...\}$$ and $$\Lambda_m^*=\overline{\mathrm{span}}\{e_{3j+m}-ae_{3j+m-1}; j=0,1,2,...\},$$
where $e_{-1}=0$ and $e_k=K_a\varphi_a^k$ for $k=0,1 ,2,...$
\end{corollary}

\begin{proof}
Let $\tau=\varphi_a\comp\varphi\comp\varphi_a$. Then $\tau$ is a rotation of order $3$, that is, $\tau(z)=\varphi'(a)z$. So $C_\tau z^k=\varphi'(a)^kz^k$ for $k=0,1,2,...$ Moreover, since $C_\tau^*=C_{\tau^{-1}}$ where $\tau^{-1}(z)=\overline{\varphi'(a)}z$, we also have $C_\tau^* z^k=\overline{\varphi'(a)}^kz^k$. Thus
$$z^k\in \mathrm{Ker}(C_\tau-\varphi'(a)^k)\cap \mathrm{Ker}(C_\tau^*-\overline{\varphi'(a)}^k)$$
for $k=0,1,2,...$

Now by the definition of $\tau$ we have $C_\varphi C_{\varphi_a}=C_\tau C_{\varphi_a}$ and $C_\varphi^* C_{\varphi_a}^*=C_\tau^* C_{\varphi_a}^*$, so
$$C_{\varphi_a}z^k\in \mathrm{Ker}(C_\varphi-\varphi'(a)^k)$$ and $$C_{\varphi_a}^*z^k\in \mathrm{Ker}(C_\varphi^*-\overline{\varphi'(a)}^k).$$ Since $\varphi'(a)^3=1$ one can get
$$\overline{\mathrm{span}}\{C_{\varphi_a}z^{3j+m}; j=0,1,2,...\}\subset\Lambda_m$$
and
$$\overline{\mathrm{span}}\{C_{\varphi_a}^*z^{3j+m}; j=0,1,2,...\}\subset\Lambda_m^*.$$
On the other hand, both $C_{\varphi_a}$ and $C_{\varphi_a}^*$ are invertible on $H^2(D)$, then
$$\overline{\mathrm{span}}\{C_{\varphi_a}z^k; k=0,1,2,...\}=\overline{\mathrm{span}}\{C_{\varphi_a}^*z^k; k=0,1,2,...\}=H^2(D).$$
so we have
$$\Lambda_m=\overline{\mathrm{span}}\{C_{\varphi_a}z^{3j+m}; j=0,1,2,...\}$$
and
$$\Lambda_m^*=\overline{\mathrm{span}}\{C_{\varphi_a}^*z^{3j+m}; j=0,1,2,...\}.$$

Finally, $C_{\varphi_a}z^k=\varphi_a^k$ and Lemma \ref{*s} shows that $C_{\varphi_a}z^k=e_k-ae_{k-1}$, hence the proof is done.
\\
\qedhere
\end{proof}

\begin{remark}
Note that $||e_j||^2=(1-|a|^2)^{-1}$ for $j=0,1,2,...$ and $\langle e_j,e_k\rangle=0$ whenever $j\ne k$.
\end{remark}

%\begin{remark}\label{re2}
%If $C_\varphi$ is $C$-symmetric with respect to a conjugation $C$, then $C\Lambda_m=\Lambda_m^*$ for $m=0,1,2$.
%\end{remark}

In what follows we list some equations that will be used in our proofs. The first one is a well known identity which can be found everywhere. One can refer to \cite{CC}, for example.  Suppose $\varphi$ is an automorphism of $D$, then
\begin{align}\label{id}
1-|\varphi(z)|^2=\frac{(1-|w|^2)(1-|z|^2)}{|1-\overline{w}z|^2}
\end{align}
for all $z\in D$. Here $w=\varphi^{-1}(0)$.

Suppose $\varphi$ is a analytic self-map of $D$ with a fixed point $a\in D$, then on $H^2(D)$ we have the following formulae,

\begin{align}\label{*1}
C_\varphi^*K_a=K_a;
\end{align}

\begin{align}\label{*2}
C_\varphi^*K_a^{[1]}=\overline{\varphi'(a)}K_a^{[1]};
\end{align}
and
\begin{align}\label{*3}
C_\varphi^*K_a^{[2]}=\overline{\varphi'(a)}^2K_a^{[2]}+\overline{\varphi''(a)}K_a^{[1]}.
\end{align}

We omit the simple calculation here.

\section{Automorphisms}

In this section, we will focus on the proof of our first main result, i.e., Theorem \ref{main}, which assert that no elliptic automorphism of order $3$ except for rotations can induce a complex symmetric composition operator on $H^2(D)$.

We would like to point out here that throughout this section, each notation will always represent the same thing as it
did initially. For example, $\varphi$ is always a elliptic automorphism of order $3$ in what follows, $a$ is always the fixed point of $\varphi$ in $D\backslash\{0\}$, and $\rho$ always represents the same constant $-\frac{\overline{a}^2}{a}\cdot\frac{1-|a|^2}{1-|a|^4}$ ever since it is introduced in Claim \ref{1}.

We will assume that $C_\varphi$ is $C$-symmetric with respect to some conjugation $C$, and finally we will show this assumption leads to a contradiction. Now we start by determining the image of a certain vector under the conjugation $C$. The notations in Lemma \ref{*s} and Corollary \ref{le} are still valid in this section.

\begin{claim}\label{1}
Let $\varphi$ be an elliptic automorphism of order $3$ with fixed point $a\in D\backslash\{0\}$. If $C_\varphi$ if $C$-symmetric on $H^2(D)$ with respect to a conjugation $C$, then we have
$$Ce_0=c_0\frac{1-\overline{a}^3\varphi_a^3}{1-\rho\varphi_a^3},$$
where $c_0$ is a constant and $$\rho=-\frac{\overline{a}^2}{a}\cdot\frac{1-|a|^2}{1-|a|^4}.$$
\end{claim}

\begin{proof}
Let $h_0=Ce_0$. Since $e_0\in\Lambda_0^*$, by Lemma \ref{zh} we have $h_0\in\Lambda_0$. So we can suppose that
$$h_0=\sum_{j=0}^\infty c_j\varphi_a^{3j}.$$

It is obvious that $e_0$ is orthogonal to $\Lambda_2^*$. So by Lemma \ref{zh} and the fact that $C$ is an isometry, $h_0$ is orthogonal to $\Lambda_2$, which means that $\langle h_0,\varphi_a^{3k+2}\rangle=0$ for $k=0,1,2,...$ So we have the following equations,
\begin{align}\label{3.1}
\sum_{j=0}^{k}c_j\overline{a}^{3k+2-3j}+\sum_{j=k+1}^{\infty}c_ja^{3j-3k-2}=0
\end{align}
for $k=0,1,2...$ Replacing $k$ by $k+1$ in (\ref{3.1}) we get
$$\sum_{j=0}^{k+1}c_j\overline{a}^{3k+5-3j}+\sum_{j=k+2}^{\infty}c_ja^{3j-3k-5}=0,$$
hence
\begin{align}\label{3.2}
\sum_{j=0}^{k+1}c_j\overline{a}^{3k+5-3j}a^3+\sum_{j=k+2}^{\infty}c_ja^{3j-3k-2}=0.
\end{align}

Combining (\ref{3.1}) and (\ref{3.2}), we have
\begin{align*}
\sum_{j=0}^kc_j\overline{a}^{3k+2-3j}+c_{k+1}a&=\sum_{j=0}^{k+1}c_j\overline{a}^{3k+5-3j}a^3,
\end{align*}
or we can write
\begin{align}\label{3.3}
\sum_{j=0}^kc_j\overline{a}^{3k+2-3j}+c_{k+1}a\frac{1-|a|^4}{1-|a|^6}=0,
\end{align}
for $k=0,1,2...$
By taking $k=1$ in (\ref{3.3}) we get
$$c_1=\tilde\rho c_0$$
where $\tilde\rho=-\frac{\overline{a}^2}{a}\cdot\frac{1-|a|^6}{1-|a|^4}$. And replacing $k$ by $k-1$ in (\ref{3.3}) we have
\begin{align}\label{3.8}
\sum_{j=0}^{k-1}c_j\overline{a}^{3k-1-3j}+c_ka\frac{1-|a|^4}{1-|a|^6}=0,
\end{align}
then (\ref{3.3}) and (\ref{3.3}) gives that
$$c_{k+1}=\rho c_k$$
for $j=1,2,3...$, where $\rho=-\frac{\overline{a}^2}{a}\cdot\frac{1-|a|^2}{1-|a|^4}$.
Therefore,
\begin{align*}
h_0&=c_0+c_1\sum_{j=1}^\infty\rho^{j-1}\varphi_a^{3j}
\\&=c_0+c_0\frac{\tilde\rho\varphi_a^3}{1-\rho\varphi_a^3}
\\&=c_0\frac{1-\overline{a}^3\varphi_a^3}{1-\rho\varphi_a^3}.
\end{align*}
\end{proof}

\begin{claim}
For the constant $c_0$ in Claim \ref{1}, we have
$$|c_0|=\frac{1}{1-|a|^4}.$$
\end{claim}

\begin{proof}
Let
$$g=\frac{\overline{\rho}-\varphi_a^3}{1-\rho\varphi_a^3},$$
then $g$ is an inner function and $g(0)=-\frac{a^2}{\overline{a}}$.

An easy calculation shows that $h_0=\gamma_1g+\gamma_2$, where
\begin{align*}
\gamma_1&=c_0\frac{\overline{a}^2}{a}(1+|a|^2);\\
\gamma_2&=c_0(1+|a|^2).
\end{align*}
So we have
\begin{align*}
||h_0||^2&=\langle\gamma_1g+\gamma_2,\gamma_1g+\gamma_2\rangle
\\&=|\gamma_1|^2+|\gamma_2|^2+2\Re\{\gamma_1\overline{\gamma_2}g(0)\}
\\&=|c_0|^2(1-|a|^2)(1+|a|^2)^2.
\end{align*}

On the other hand, since $C$ is isometric, we can know that
$$||h_0||^2=||e_0||^2=\frac{1}{1-|a|^2},$$
thus $|c_0|=(1-|a|^4)^{-1}$.
\\
\qedhere
\end{proof}

\begin{claim}\label{3}
For the function $h_0=Ce_0$ in the proof of Claim \ref{1}, we have
$$\langle h_0,\varphi^{3k}\rangle=c_0(1-|a|^4)\rho^k$$\
for $k=0,1,2...$
\end{claim}

\begin{proof}
\begin{align}\label{3.4}
\langle h_0,\varphi^{3k}\rangle=\sum_{j=0}^{k}c_j\overline{a}^{3k-3j}+\sum_{j=k+1}^{\infty}c_ja^{3j-3k}.
\end{align}
Comparing (\ref{3.4}) with (\ref{3.1}), we can get that
\begin{align}\label{3.5}
\langle h_0,\varphi^{3k}\rangle=(1-|a|^4)\sum_{j=0}^{k}c_j\overline{a}^{3k-3j}.
\end{align}
So (\ref{3.5}), along with (\ref{3.3}), shows that
\begin{align*}
\langle h_0,\varphi^{3k}\rangle&=(1-|a|^4)\frac{c_{k+1}}{\tilde\rho}
\\&=c_0(1-|a|^4)\rho^k.
\end{align*}
\end{proof}

\begin{claim}\label{4}
Under the assumption of Claim \ref{1}, we have
$$Ce_1=-c_0\frac{\overline{a}(1-|a|^6)}{a(1-|a|^4)}\cdot\frac{\varphi_a(1-\overline{a}^3\varphi_a^3)}{(1-\rho\varphi_a^3)^2}+\overline{a}h_0.$$
\end{claim}

\begin{proof}
Let $h_1=Ce_1$. Since $e_1-ae_0\in\Lambda_1^*$, by Lemma \ref{zh} we have $h_1-\overline{a}h_0\in\Lambda_1$. So we can assume that
$$h_1=\sum_{j=0}^\infty b_j\varphi_a^{3j+1}+\overline{a}h_0.$$

It is obvious that $e_1$ is orthogonal to $\Lambda_0^*$, so $h_1$ is orthogonal to $\Lambda_0$, which means that $\langle h_1,\varphi_a^{3k}\rangle=0$ for $k=0,1,2...$ So by using Claim \ref{3} we have the following equations,
\begin{align}\label{new1}
\sum_{j=0}^{\infty}b_ja^{3j+1}+c_0\overline{a}(1-|a|^4)=0,
\end{align}
and
\begin{align}\label{new2}
\sum_{j=0}^{k-1}b_j\overline{a}^{3k-3j-1}+\sum_{j=k}^{\infty}b_ja^{3j+1-3k}+c_0\overline{a}(1-|a|^4)\rho^k=0,
\end{align}
for $k=1,2,3...$

By taking $k=1$ in (\ref{new2}) we get
$$b_0\overline{a}^2+\sum_{j=1}^{\infty}b_ja^{3j-2}+c_0\overline{a}(1-|a|^4)\rho=0.$$
Comparing this equation with (\ref{new1}) we have
\begin{align}\label{new3}
b_0a\frac{1-|a|^4}{1-|a|^6}+c_0\overline{a}=0.
\end{align}
Replacing $k$ by $k+1$ in (\ref{new2}) we get
$$\sum_{j=0}^{k}b_j\overline{a}^{3k-3j+2}+\sum_{j=k+1}^{\infty}b_ja^{3j-2-3k}+c_0\overline{a}(1-|a|^4)\rho^{k+1}=0.$$
Combining this equation with (\ref{new2}) we have
\begin{align}\label{new4}
\sum_{j=0}^{k-1}b_j\overline{a}^{3k-3j-1}+b_ka\frac{1-|a|^4}{1-|a|^6}+c_0\overline{a}\rho^k=0
\end{align}
for $k=1,2,3...$

Now let
$$\delta_j=-\frac{b_ja(1-|a|^4)}{c_0\overline{a}(1-|a|^6)},$$
then (\ref{new3}) and (\ref{new4}) shows that $\delta_0=1$, $\delta_1=\rho+\tilde\rho$, and
$$\delta_{k+1}=\rho\delta_k+\tilde\rho\rho^k$$
for $k=1,2,3...$ Hence we can get that
$$\delta_k=\rho^k+k\tilde\rho\rho^{k-1}$$
for $k=0,1,2...$

Thus we have
\begin{align*}
h_1-\overline{a}h_0&=\sum_{j=0}^\infty b_j\varphi_a^{3j+1}
\\&=-c_0\frac{\overline{a}(1-|a|^6)}{a(1-|a|^4)}\sum_{j=0}^\infty \delta_j\varphi_a^{3j+1}
\\&=-c_0\frac{\overline{a}(1-|a|^6)}{a(1-|a|^4)}\left(\sum_{j=0}^\infty\rho^j\varphi_a^{3j+1}+\sum_{j=0}^\infty j\tilde\rho\rho^{j-1}\varphi_a^{3j+1}\right)
\\&=-c_0\frac{\overline{a}(1-|a|^6)}{a(1-|a|^4)}\left(\frac{\varphi_a}{1-\rho\varphi_a^3}+\frac{\tilde\rho\varphi_a^4}{(1-\rho\varphi_a^3)^2}\right)
\\&=-c_0\frac{\overline{a}(1-|a|^6)}{a(1-|a|^4)}\cdot\frac{\varphi_a(1-\overline{a}^3\varphi_a^3)}{(1-\rho\varphi_a^3)^2}.
\end{align*}
\end{proof}

Now we can prove our final result as follows.

\begin{theorem}\label{main}
If $\varphi$ is an elliptic automorphism of order 3 with fixed point $a\in D\backslash \{0\}$, then $C_\varphi$ is not complex symmetric on $H^2(D)$.
\end{theorem}

\begin{proof}
Suppose that $C_\varphi$ is $C$-symmetric with respect to conjugation $C$. Then Claims \ref{1} to \ref{4} hold.

Let
$$f=\frac{1-|a|^6}{1-|a|^4}\cdot\frac{1-\overline{a}^3\varphi_a^3}{(1-\rho\varphi_a^3)^2},$$
Then
$$||f||=|c_0|^{-1}\cdot||h_1-\overline{a}h_0||=(1-|a|^4)||h_1-\overline{a}h_0||.$$

A tedious calculation shows that $f=\beta_1g^2+\beta_2g+\beta_3$, where $$g=\frac{\overline{\rho}-\varphi_a^3}{1-\rho\varphi_a^3},$$ and
\begin{align*}
\beta_1&=\frac{(1-|a|^4)(1+|a|^2)^2}{1-|a|^6}(\rho^2-\overline{a}^3\rho)=\frac{\overline{a}^4}{a^2}(1+|a|^2);
\\\beta_2&=\frac{(1-|a|^4)(1+|a|^2)^2}{1-|a|^6}(-2\rho+\overline{a}^3+\overline{a}^3|\rho|^2)=\frac{\overline{a}^2}{a}(1+|a|^2)(2+|a|^2);
\\\beta_3&=\frac{(1-|a|^4)(1+|a|^2)^2}{1-|a|^6}(1-\overline{a}^3\overline{\rho})=(1+|a|^2)^2.
\end{align*}

So
\begin{align*}
||f||^2&=\langle\beta_1g^2+\beta_2g+\beta_3,\beta_1g^2+\beta_2g+\beta_3\rangle
\\&=|\beta_1|^2+|\beta_2|^2+|\beta_3|^2+2\Re\left(\beta_1\overline{\beta_2}g(0)+\beta_2\overline{\beta_3}g(0)+\beta_1\overline{\beta_3}g(0)^2\right)
\\&=(1+2|a|^2-2|a|^4-|a|^6)(1+|a|^2)^2.
\end{align*}

On the other hand, since $C$ is isometric,
\begin{align*}
||f||^2&=(1-|a|^4)^2||h_1-\overline{a}h_0||^2
\\&=(1-|a|^4)^2||e_1-\overline{a}e_0||^2
\\&=(1-|a|^4)^2\frac{1+|a|^2}{1-|a|^2}
\\&=(1-|a|^4)(1+|a|^2)^2.
\end{align*}
So we have
\begin{align*}
(1+2|a|^2-2|a|^4-|a|^6)(1+|a|^2)^2&=(1-|a|^4)(1+|a|^2)^2
\\2|a|^2-|a|^4-|a|^6&=0
\\|a|^2+|a|^4&=2,
\end{align*}
which is impossible since $a\in D$.
\\
\qedhere
\end{proof}

As a corollary we can get our Main Result A. It gives a complete description of the automorphisms which can induce complex symmetric operators on $H^2(D)$.

\begin{corollary}\label{last}
Suppose $\varphi$ is an automorphism of $D$. Then $C_\varphi$ is complex symmetric on $H^2(D)$ if and only if $\varphi$ is either a rotation or an elliptic automorphism of order two.
\end{corollary}

\section{Non-automorphisms}

In this section we will consider the non-automorphic cases and give the proof of our Main Result B. According to Proposition \ref{propa}, we will always assume that $\varphi\in LFT(D)$ has a fixed point in $D$.

\subsection{linear fractional transformations that fix $0$}

First we shall consider the linear fractional transformations who fix the point $0$. So throughout this subsection we assume $\varphi$ is of the form
$$\varphi(z)=\frac{bz}{1-cz}$$
for $z\in D$.

\begin{remark}
It is easy to check that $\varphi(z)=\frac{bz}{1-cz}$ is a self-map of $D$ if and only if $|b|+|c|\leqslant 1$. Moreover, we may assume that $\varphi$ is neither a constant nor a linear transform on $D$. This means neither $b$ nor $c$ is zero.
\end{remark}

The next Lemma gives the solution of the Schroeder's equation of such $\varphi$.

\begin{lemma}\label{tz}
Suppose $b,c\ne 0$, $|b|+|c|\leqslant 1$, and $\varphi(z)=bz/(1-cz)$. Let
$$\sigma(z)=\frac{z}{1-\eta z},$$
where $\eta=c/(1-b)$.
Then $\sigma\comp\varphi=b\sigma$.
\end{lemma}

\begin{proof}
Just a simple and direct calculation.
\\
\qedhere
\end{proof}

\begin{remark}
Note that $1/\eta=(1-b)/c$ is the other fixed point of $\varphi(z)=bz/(1-cz)$ on $\hat{\mathbb{C}}$ except for $0$, so it lies outside the unit disk $D$.
\end{remark}

As a corollary of Lemma \ref{tz}, the next result shows that if $C_\varphi$ is complex symmetric, then $1/\eta$ can never belong to the unit circle $\partial D$. In fact, it is actually a part of Theorem 2.10 in \cite{Arxiv}. However, we still present a different proof here, because this proof will be used in Proposition \ref{31} in the next subsection.

\begin{corollary}\label{yg}
Suppose $\varphi\in LFT(D)$ is not an automorphism and suppose that $\varphi(0)=0$. If $C_\varphi$ is complex symmetric on $H^2(D)$, then $0$ is the only fixed point of $\varphi$ in $\overline{D}$.
\end{corollary}

\begin{proof}
Assume that $\varphi(z)=bz/(1-cz)$. Then the two fixed points of $\varphi$ on $\hat{\mathbb{C}}$ are $0$ and $(1-b)/c$.

Since $\varphi(0)=0$, (\ref{*2}) shows that $\overline{b}=\overline{\varphi'(0)}$ is an eigenvalue of $C_\varphi^*$. Then the complex symmetry of $C_\varphi$ implies that $b$ is an eigenvalue of $C_\varphi$. However, Theorem 2.63 in \cite{CC}  and Lemma \ref{tz} show that each eigenvector for $C_\varphi$ corresponding to the eigenvalue $b$ can only be a constant multiple of $\sigma(z)=z/(1-\eta z)$ where $\eta=c/(1-b)$. So we must have $\sigma\in H^2(D)$, which means that $|\eta|<1$ since
$$\sigma(z)=\sum_{j=0}^\infty \eta^j z^{j+1}.$$
Therefore, $(1-b)/c=1/\eta$, as a fixed point of $\varphi$, lies outside $\overline{D}$.
\\
\qedhere
\end{proof}

\subsection{linear fractional transformations that fix $a\in D\backslash \{0\}$}

Now we turn to the general cases. We will assume in this subsection that $\varphi$ is a linear fractional map with a fixed point in $D$ other than $0$. The next proposition follows from the proof of Corollary \ref{yg}.

\begin{proposition}\label{31}
Suppose $\varphi\in LFT(D)$ is not an automorphism and $a\in D\backslash \{0\}$ is a fixed point of $\varphi$. If $C_\varphi$ is complex symmetric on $H^2(D)$, then $a$ is the only fixed point of $\varphi$ in $\overline{D}$.
\end{proposition}

\begin{proof}
Let $\tilde\varphi=\varphi_a\comp\varphi\comp\varphi_a$. Then $\tilde\varphi(0)=0$. We can assume that $\tilde\varphi(z)=bz/(1-cz).$ Note that $b=\varphi'(a)$, and the other fixed point of $\varphi$ on $\hat{\mathbb{C}}$ except for $a$ is $\varphi_a(\frac{1-b}{c})$. Again by (\ref{*2}) we can know that $\overline{b}=\overline{\varphi'(a)}$ is an eigenvalue of $C_\varphi^*$. So the complex symmetry of $C_\varphi$ shows that $b=\varphi'(a)$ is an eigenvalue of $C_\varphi$. By the proof of Corollary \ref{yg}, we have $|(1-b)/c|>1$. Therefore $\varphi_a(\frac{1-b}{c})$ lies outside $\overline{D}$ since $\varphi_a$ is an automorphism of $D$.
\\
\qedhere
\end{proof}

The next Theorem shows that if the fixed point of $\varphi$ in $D$ is not zero and $C_\varphi$ is complex symmetric, then the fixed point of $\varphi$ outside $\overline{D}$ must be $\infty$.

\begin{theorem}\label{final}
Suppose $\varphi\in LFT(D)$ is not an automorphism and $a\in D\backslash \{0\}$ is a fixed point of $\varphi$. If $C_\varphi$ is complex symmetric on $H^2(D)$, then $\varphi$ is a polynomial of degree one with respect to $z$.
\end{theorem}

\begin{proof}
Let $\tilde\varphi=\varphi_a\comp\varphi\comp\varphi_a$. Then $\tilde\varphi(0)=0$, hence $\tilde\varphi$ is of the form
$$\tilde\varphi(z)=\frac{bz}{1-cz}.$$
Proposition \ref{31} implies that $\eta=c/(1-b)$ lies in $D$. So $\sigma(z)=z/(1-\eta z)$ belongs to $H^2(D)$.

Since $\varphi$ is not an automorphism, Lemma \ref{tz} shows that $\mathrm{Ker}(C_{\tilde\varphi}-b^j)$ is a subspace of dimension one spanned by $\sigma^j$. On the other hand, by using (\ref{*1}), (\ref{*2}) and (\ref{*3}) one can check that $f_j\in \mathrm{Ker}(C_{\tilde\varphi}^*-\overline{b}^j)$ for $j=0,1,2$. Here $f_0=1$, $f_1=z$, $f_2=z^2-\overline\eta z$.

Now let $h_j=(1-\eta a)^{-1}C_{\varphi_a}\sigma^j$, Since $C_\varphi C_{\varphi_a}=C_{\varphi_a}C_{\tilde\varphi}$, we can conclude that $\mathrm{Ker}(C_\varphi-b^j)$ is a subspace of dimension one spanned by $h_j$. Similarly, we can know that $C_{\varphi_a}^*f_j\in \mathrm{Ker}(C_\varphi^*-\overline{b}^j)$ for $j=1,2,3$. A simple calculation shows that
$$h_j=\left(\frac{a-z}{1-w_0z}\right)^j,$$
where
$$w_0=\varphi_{\overline{a}}(\eta)=\frac{\overline{a}-\eta}{1-a\eta}.$$
Also Lemma \ref{*s} shows that
\begin{align*}
C_{\varphi_a}^*f_0&=e_0;
\\C_{\varphi_a}^*f_1&=e_1-ae_0;
\\C_{\varphi_a}^*f_2&=e_2-(a+\overline{\eta})e_1+\overline{\eta}ae_0,
\end{align*}
where $e_j=K_a\varphi_a^j$ for $j=1,2,3$.

Since $C_\varphi$ is complex symmetric, say $C_\varphi C=CC_\varphi^*$ for some conjugation $C$, Lemma \ref{zh} shows that there exist $\lambda_0, \lambda_1, \lambda_2\in\mathbb{C}$ such that $$Ch_j=\lambda_jC_{\varphi_a}^*f_j$$ for $j=0,1,2$.

Conjugation $C$ is an isometry, so
$$\left|\langle Ch_0,Ch_1\rangle\right|=|\langle h_0,h_1\rangle|=|a|.$$
Whence
\begin{align*}
|a|&=|\langle \lambda_0C_{\varphi_a}^*f_0,\lambda_1C_{\varphi_a}^*f_1\rangle|
\\&=|\lambda_0\lambda_1|\cdot|a|\cdot||e_0||^2
\\&=|\lambda_0\lambda_1|\frac{|a|}{1-|a|^2},
\end{align*}
so $|\lambda_0\lambda_1|=1-|a|^2$. However,
\begin{align*}
|\lambda_0|&=\frac{||Ch_0||}{||C_{\varphi_a}^*f_0||}=\frac{||h_0||}{||e_0||}=\sqrt{1-|a|^2}.
\end{align*}
Therefore $|\lambda_1|=\sqrt{1-|a|^2}$. Then we have
\begin{align}\label{fz}
||h_1||^2&=|\lambda_1|^2||C_{\varphi_a}^*f_1||^2\nonumber
\\&=(1-|a|^2)\frac{1+|a|^2}{1-|a|^2}\nonumber
\\&=1+|a|^2.
\end{align}
On the other hand, we can write $$h_1=\varphi_{\overline{w_0}}+(a-\overline{w_0})K_{\overline{w_0}}.$$ So
\begin{align}\label{fy}
||h_1||^2&=||\varphi_{\overline{w_0}}||^2+|a-\overline{w_0}|^2||K_{\overline{w_0}}||^2\nonumber
\\&=1+\frac{|a-\overline{w_0}|^2}{1-|w_0|^2}.
\end{align}

Combining (\ref{fz}) and (\ref{fy}), we have
$$|a|^2=\frac{|a-\overline{w_0}|^2}{1-|w_0|^2}.$$
Notice that $w_0=\varphi_{\overline{a}}(\eta)$, which also means $\varphi_{\overline{a}}(w_0)=\eta$. So by using identity (\ref{id}) we can get
\begin{align*}
|a|^2&=\frac{|a-\overline{w_0}|^2}{1-|w_0|^2}
\\&=|\overline{a}-w_0|^2\frac{1-|a|^2}{(1-|\eta|^2)|1-aw_0|^2}
\\&=\frac{1-|a|^2}{1-|\eta|^2}|\eta|^2.
\end{align*}
Thus we can know that $|\eta|=|a|$.

Now we turn to investigate the relationship between $h_2$ and $C_{\varphi_a}^*f_2$. It would be helpful if one could note that
$$\langle C_{\varphi_a}^*f_2+\overline{\eta}C_{\varphi_a}^*f_1,C_{\varphi_a}^*f_0\rangle=0$$
and
$$\langle h_2-ah_1,h_0\rangle=0.$$
Therefore, again by Lemma \ref{zh} we can find $\lambda\in\mathbb{C}$ such that
$$C(h_2-ah_1)=\lambda \left(C_{\varphi_a}^*f_2+\overline{\eta}C_{\varphi_a}^*f_1\right),$$
or we can write
$$\lambda_2C_{\varphi_a}^*f_2-\overline{a}\lambda_1C_{\varphi_a}^*f_1=\lambda C_{\varphi_a}^*f_2+\lambda\overline{\eta}C_{\varphi_a}^*f_1.$$
So $-\overline{a}\lambda_1=\lambda\overline{\eta}$, hence $$|\lambda|=|\lambda_1|=\sqrt{1-|a|^2}.$$ Thus we have
\begin{align}\label{f2z}
||h_2-ah_1||^2&=|\lambda|^2\cdot||C_{\varphi_a}^*f_2+\overline{\eta}C_{\varphi_a}^*f_1||^2\nonumber
\\&=(1-|a|^2)||e_2-ae_1||^2\nonumber
\\&=1+|a|^2.
\end{align}

On the other hand,
\begin{align*}
h_2(z)-ah_1(z)=(1-aw_0)\frac{z(z-a)}{(1-w_0z)^2}.
\end{align*}
Let $$\tilde{h}=(1-aw_0)\frac{z-a}{(1-w_0z)^2},$$
then $h_2(z)-ah_1(z)=z\tilde{h}(z)$, hence $||h_2-ah_1||=||\tilde{h}||$. A calculation shows that
$$\tilde{h}=\gamma_1K_{\overline{w_0}}+\gamma_2 K_{\overline{w_0}}\varphi_{\overline{w_0}},$$
where
\begin{align*}
\gamma_1&=-\frac{(1-aw_0)(a-\overline{w_0})}{1-|w_0|^2};\\
\gamma_2&=-\frac{(1-aw_0)^2}{1-|w_0|^2}.
\end{align*}
Again by using (\ref{id}) and noticing that $$|\eta|=\left|\frac{a-\overline{w_0}}{1-aw_0}\right|,$$ we have
\begin{align*}
|\gamma_1|&=\frac{|1-aw_0||a-\overline{w_0}|}{1-|w_0|^2}
\\&=|\eta|\frac{|1-aw_0|^2}{1-|w_0|^2}
\\&=|\eta|\frac{1-|a|^2}{1-|\eta|^2}=|a|,
\end{align*}
and
\begin{align*}
|\gamma_2|&=\frac{|1-aw_0|^2}{1-|w_0|^2}
\\&=\frac{1-|a|^2}{1-|\eta|^2}=1.
\end{align*}
So
\begin{align}\label{f2y}
||h_2-ah_1||^2&=(|\gamma_1|^2+|\gamma_2|^2)\frac{1}{1-|w_0|^2}\nonumber
\\&=\frac{1+|a|^2}{1-|w_0|^2}.
\end{align}
Combining (\ref{f2z}) with (\ref{f2y}) we get $w_0=0$, which means that $\eta=\overline{a}$. Thus
$$\varphi_a(z)=\frac{\overline{\eta}-z}{1-\eta z}.$$
So $\varphi_a$, as a member in $LFT(\hat{\mathbb{C}})$, maps $1/\eta$ to $\infty$. Note that $1/\eta$ is one of the fixed points of $\tilde\varphi$, therefore $\infty$ is a fixed point of $\varphi=\varphi_a\comp\tilde\varphi\comp\varphi_a$. Thus $\varphi$ is a polynomial of degree one.
\\
\qedhere
\end{proof}

As a conclusion, we can now get our Main Result B as follows.

\begin{theorem}\label{main2}
Suppose $\varphi\in LFT(D)$ is not a constant. Then $C_\varphi$ is complex symmetric on $H^2(D)$ if and only if at least one of the following conditions holds,
\\
(i) the two fixed points of $\varphi$ on $\hat{\mathbb{C}}$ are $0$ and a point outside $\overline{D}$;
\\
(ii) the two fixed points of $\varphi$ on $\hat{\mathbb{C}}$ are $\infty$ and a point in $D$;
\\
(iii) $\varphi$ is an involutive automorphism.
\end{theorem}

\begin{proof}
If $\varphi$ is an automorphism of $D$, then by Corollary \ref{last}, $\varphi$ satisfies (i) or (iii).

If $\varphi$ is not an automorphism of $D$ and $\varphi(0)=0$, then by Proposition \ref{he}, $\varphi$ satisfies (i).

If $\varphi$ is not an automorphism of $D$ and the fixed point of $\varphi$ in $D$ is not zero, then Theorem \ref{final} shows that $\infty$ is a fixed point of $\varphi$. Thus by Proposition \ref{he}, $\varphi$ satisfies (ii).
\\
\qedhere
\end{proof}

The theorems in this paper give a complete description of complex symmetric composition operators whose symbols are linear fractional. However, people know little about the case when $\varphi$ is not linear fractional. Even no positive example has been found. We list it here as an question.

\begin{open}
Is there any complex symmetric composition operator whose symbol is not linear fractional?
\end{open}


\begin{thebibliography}{99}

\bibitem{CC} C. C. Cowen, B. D. MacCluer, Composition Operators on Spaces of Analytic Functions, CRC Press, Boca Raton, 1995.

\bibitem{BN} P. S. Bourdon, S. Waleed Noor, Complex symmetry of invertible composition operators, J. Math. Anal. Appl. 429 (2015) 105-110.

\bibitem{JFA} S. Jung, Y. Kim, E. Ko, J.E. Lee, Complex symmetric weighted composition operators on $H^2(D)$, J. Funct. Anal. 265 (2014) 323-351

\bibitem{Arxiv} S. K. Narayan, D. Sievewright, D. Thompson, Complex symmetric composition operators on $H^2$, J. Math. Anal. Appl. 443 (2016) 625-630.

\bibitem{GH} S. R. Garcia, C. Hammond, Which weighted composition operators are complex symmetric?, Oper. Theory Adv. Appl. 236 (2014) 171-179.

\bibitem{CS1} S. R. Garcia, M. Putinar, Complex symmetric operators and applications, Trans. Amer. Math. Soc. 358 (3) (2006) 1285-1315.

\bibitem{Cs2} S. R. Garcia, M. Putinar, Complex symmetric operators and applications. II, Trans. Amer. Math. Soc. 359 (8) (2007) 3913-3931.

\bibitem{Cs3} S. R. Garcia, W. R. Wogen, Some new classes of complex symmetric operators, Trans. Amer. Math. Soc. 362 (11) (2010) 6065-6077.

\bibitem{Invo} S. Waleed Noor, Complex symmetry of composition operators induced by involutive ball automorphisms, Proc. Amer. Math. Soc. 142 (2014) 3103-3107.

\bibitem{Cor} S. Waleed Noor, On an example of a complex symmetric composition operator on $H^2(D)$, J. Funct. Anal. 269 (2015) 1899-1901.
\end{thebibliography}
\end{document}